\documentclass[12pt, a4paper]{article}
\usepackage[cp1251]{inputenc}
\usepackage[english]{babel}
\usepackage{amssymb,amsmath,amsthm,amscd}

\renewcommand\refname{\centering REFERENCES}

\theoremstyle{plain}
\newtheorem{thm}{Theorem}
\newtheorem{lem}{Lemma}
\newtheorem{cor}{Corollary}

\theoremstyle{definition}
\newtheorem{rem}{Remark}

\numberwithin{equation}{section}

\numberwithin{thm}{section}

\numberwithin{lem}{section}

\numberwithin{cor}{section}

\numberwithin{rem}{section}

\begin{document}
	
	\renewcommand\refname{\centering REFERENCES}
	
	\renewcommand\proofname{\centering Proof}

	\begin{center}
		{\Large\bf On  the Commutant of 
			the Generalized Backward Shift Operator
			in Weighted Spaces of Entire Functions}
	\end{center}

	\begin{center}
		{\Large Olga~A.~Ivanova, Sergej~N.~Melikhov}
	\end{center}
	
	\bigskip
	
	{\small {\bf Abstract.} We investigate continuous linear operators, which commute 
		with the generalized backward shift operator
		(a one-dimensional perturbation of the Pommiez operator) in a countable inductive limit
		$E$ of weighted Banach spaces of entire functions. This space $E$ is isomorphic with the help of
		the Fourier-Laplace transform to the strong dual of the Fr\'echet space of all holomorphic
		functions on a convex domain $Q$ in the complex plane, containing the origin. 
		Necessary and sufficient conditions are obtained that an operator of the mentioned commutant
		is a topological isomorphism of $E$. The problem of the factorization of nonzero operators of this commutant is investigated.
		In the case when the function, defining the generalized backward shift operator, has zeros in $ Q $, they are divided into two classes:
		the first one consists of isomorphisms and surjective operators with a finite-dimensional kernel,
and the second one contains finite-dimensional operators.
		Using obtained results, we study the generalized Duhamel product in
		Fr\'echet space of all holomorphic functions on $Q$.
		
		
	}
	
	\bigskip

%
%
%
%
%
%
%
%
%

\section{Introduction}

Let $Q$ be a convex domain in $\mathbb C$, containing the origin; $H(Q)$ be the Fr\'echet space of all
holomorphic functions on $Q$; $E$ be a countable inductive limit of weighted Banach spaces which with the help of
Fourier-Laplace transform is topologically isomorphic to the strong dual of $H(Q)$. A function $g_0\in E$ satisfying the condition
$g_0(0)=1$
defines the generalized backward shift operator $D_{0,g_0}(f)(t)=\frac{f(t)-g_0(t)f(0)}{t}$, which is continuous
and linear in $E$. If $g_0\equiv 1$, then $D_{0,g_0}$ is the usual backward shift operator
(Pommiez operator) $D_0$. In the general case $D_{0,g_0}$ is a one-dimensional perturbation of $D_0$.

The problem, which we solve in this article, is to investigate the structure of the set $\mathcal K(D_{0,g_0})$
of all continuous linear operators in $E$, which commute with $D_{0,g_0}$ in $E$. The set
$\mathcal K(D_{0,g_0})$ has been described in \cite{AA}. In main
results we assume that the 
function $g_0$ has a finite number of zeros or has no zeros, i.\,e.,
$g_0(z)=P(z) e^{\lambda z}$ for some $\lambda\in Q$ and some polynomial $P$, such that $P(0)=1$. 
In Theorems \ref{ISOMORPHISM} and \ref{RASCHEPL} it is shown, that $\mathcal K(D_{0,g_0})$ is divided into two classes.
The first one consists of isomorphisms and surjective operators with a finite-dimensional kernel,
and the second one contains finite-dimensional operators. If $g_0$ has no zeros, i.\,e., $g_0(z)=e^{\lambda z}$
for some $\lambda\in Q$, then the second class is empty.
Previously V.A. Tkachenko \cite{TKACH} investigated properties of the commutant of the operator of generalized integration in
a space of analytic functionals. This space is the dual of a countable inductive limit of weighted Banach
spaces of entire functions, the growth of which is defined by a $\rho$-trigonometrically convex function
($\rho>0$). The operator of generalized integration is the adjoint map
of $D_{0,g_0}$, defined by the function $g_0=e^{ \mathcal P}$ for some polynomial $\mathcal P$. Such function $g_0$
has no zeros.

In the dual $E'$ of $E$ shift operators for $D_{0,g_0}$ define a product $\otimes$
by the convolution rule. If we identify the strong dual of $E$ with $H(Q)$ with the help of the adjoint
map of the Fourier-Laplace transform, the operation $\otimes$ is realized in  $H(Q)$ 
as the generalized Duhamel product. In the case of $g_0\equiv 1$ it coincides with the Duhamel product
(with the derivative of the Mikusinski convolution product). The Duhamel product
is closely related to the Volterra operator. It is being studied quite intensively
(see the paper of M.T. Karaev \cite{KARAEV}). This multiplication is used in the theory of ordinary
differential equations with constant coefficients, in the boundary value problem
of mathematical physics
(in the sloping beach problem), in the spectral theory of direct sums of operators.
Investigations of the Duhamel product in the space of all holomorphic functions on a domain in $\mathbb C$
go back to N. Wigley \cite{WIGLEY}.
In this article we prove the criterion that the multiplication operator, which is defined by generalized
Duhamel product, is an isomorphism of $H(Q)$.

Note that the situation, when $g_0$ has zeros, differ significantly from one when $g_0$ has no zeros.
Namely, our proofs use essentially the description of the lattice of proper closed $ D_{0, g_0} $-invariant
subspaces of $E$, obtained in \cite{ITOGI17}. If $g_0$ has no zeros, $D_{0,g_0}$
is unicellular. If $g_0$ has zeros, then 
this lattice is not linearly ordered, 
moreover, the family of finite-dimensional closed $D_{0, g_0}$-invariant subspaces of $H(Q)$
is also not linearly ordered. 
The mixed structure of this lattice implies the existence of two "extreme" \, 
subsets of $\mathcal K(D_{0, g_0})$.


\section{Preliminary information}

Let $Q$ be a convex domain in $\mathbb C$, containing the origin; $(Q_n)_{n\in\mathbb N}$ be a sequence
of convex compact subsets of $Q$, such that $Q_n\subset{\rm int}\,Q_{n+1}$, $n\in\mathbb N$,
and $Q=\bigcup\limits_{n\in\mathbb N} Q_n$. The symbol ${\rm int}\,M$ denotes  the interior of a set $M\subset\mathbb C$
in $\mathbb C$. For a bounded set $M\subset\mathbb C$ let $H_M$ be
the support function of $M$: $H_M(z):=\sup\limits_{t\in\mathbb C}{\rm Re}(zt)$, $z\in\mathbb C$.
We set $H_n:=H_{Q_n}$, $n\in\mathbb N$.

Define weighted Banach spaces 
$$
E_{n}:=\left\{ f\in H(\mathbb C)\,|\,
\|f\|_{n}:=\sup\limits_{z\in\mathbb C}\frac{|f(z)|}{\exp(H_n(z))}<+\infty\right\}, \,n\in\mathbb N.
$$
Here $H(\mathbb C)$ is  the space of all entire functions on
$\mathbb C$.
Note that $E_n$ is embedded continuously in $E_{n+1}$ for each 
$n\in \mathbb N$.
Put $E:=\bigcup\limits_{n\in\mathbb N} E_{Q,n}$ and we 
endow $E$ with the
topology of the inductive limit of the sequence of Banach spaces 
$E_n$, $n\in\mathbb N$, with respect to embeddings $E_n$ in
$E$ (see \cite[Ch.~III, \S~24]{MEIVOGT}): \, $E:=\mathop{\rm ind}\limits_{n
	\rightarrow} E_n$.

Let $H(Q)$ be the space of all holomorphic functions on $Q$ with the compact convergence topology.
For a locally convex space $F$ we denote by $F'$
the dual of $F$. We put $e_{z}(t):=e^{z t}$, $z, t\in\mathbb C$.
The Fourier-Laplace transform $\mathcal F(\varphi)(z):=\varphi(e^{z})$, $z\in\mathbb C$,
$\varphi\in H(Q)'$, is a topological isomorphism of the strong dual of $H(Q)$ on $E$
\cite[Theorem 4.5.3]{HERM}.

Fix a function $g_0\in E$ with $g_0(0)=1$. The generalized backward shift operator is defined by
$$
D_{0,g_0}(f)(t):=\left\{
\begin{array}{cc}
\frac{f(t)-g_0(t)f(0)}{t}, & t\ne 0,\\
f'(0)-g_0'(0)f(0), & t=0,
\end{array}
\right.
$$
$f\in E$.
Following \cite{BINDERMAN}, \cite{Dimovsky}, we introduce
{\it shift operators for the operator} $D_{0,g_0}$ 
$$
T_{z,g_0}(f)(t):=\left\{
\begin{array}{cc}
\frac{tf(t)g_0(z)-zf(z)g_0(t)}{t-z}, & \, t\ne z,\\
zg_0(z)f'(z)-zf(z)g'_0(z)+f(z)g_0(z), & t=z,
\end{array}
\right.
$$
$z\in\mathbb C$,
$f\in E$. Set
$$
\widetilde T_{z,g_0}(f)(t):=\left\{
\begin{array}{cc}
\frac{f(t)g_0(z)-f(z)g_0(t)}{t-z}, & \, t\ne z,\\
g_0(z)f'(z)-f(z)g'_0(z), & t=z,
\end{array}
\right.
$$
$f\in E$, $z\in\mathbb C$.
For $z\in\mathbb C$ the Pommiez operators $D_z$ 
are defined by
$$
D_z(f)(t):=\left\{
\begin{array}{cc}
\frac{f(t)-f(z)}{t-z}, & t\ne z,\\
f'(z),& t=z,
\end{array}
\right.
$$
$f\in E$.
All operators $T_{z,g_0}$, $\widetilde T_{z,g_0}$, $D_z$, $z\in\mathbb C$, 
are continuous and linear in $E$.

For an integer $n\ge 0$ by $\mathbb C[z]_n$ we denote the space of polynomials of degree
at most $n$. Note that ${\rm Ker}\,D_{0,g_0}^n=g_0\mathbb C[z]_{n-1}$ for all 
$n\in\mathbb N$.

With the help of shift operators for $D_{0,g_0}$ in $E'$ one can define a multiplication $\otimes$
by $(\varphi\otimes\psi)(f)=\varphi_z(\psi(T_{z,g_0}(f))$, \,$\varphi, \psi\in E'$, $f\in E$.
By  \cite[2.2]{AA} the space $E'$ is an associative and
commutative algebra with the multiplication $\otimes$.

Let $\mathcal K(D_{0,g_0})$ be the set of all continuous linear operators $B$ in $E$, such that
$BD_{0,g_0}=D_{0,g_0} B$ in $E$. It is an algebra with composition of the operators with the 
rule of the multiplication. Note that $T_{z,g_0}\in \mathcal K(D_{0,g_0})$ for every $z\in\mathbb C$.
For a functional $\varphi\in E'$ we define the operator 
$B_\varphi(f)(z):=\varphi\left(T_{z,g_0}(f)\right)$, $z\in\mathbb C$, $f\in E$.
It is continuous and linear in $E$. 

In \cite[Lemma 17]{AA} the following result is proved:
\begin{thm} \label{ALGEBRA}
	The map $\varphi\mapsto B_\varphi$ is an isomorphism of the algebra $(E',\otimes)$ onto
	$\mathcal K(D_{0,g_0})$.
\end{thm}
From Theorem \ref{ALGEBRA} it follows that the algebra $\mathcal K(D_{0,g_0})$
is commutative.

From the commutativity of $\otimes$ it follows that for each $\varphi\in E'$ the
convolution operator $S_\varphi:E'\to E'$, $\psi\mapsto\varphi\otimes\psi$, is the adjoint
map of $B_\varphi:E\to E$ with respect to dual system $(E, E')$.

\begin{rem} \label{supportfunction}
	We will use the following well known properties of support functions $H_n$:
	
	\begin{itemize}
		\item[(i)] For each $n\in\mathbb N$ there is $\varepsilon>0$, such that
		$$
		\sup\limits_{z\in\mathbb C}\left(H_n(z) +\varepsilon|z|- H_{n+1}(z)\right)<+\infty.
		$$
				
		\item[(ii)] 
		$\lim\limits_{|t|\to+\infty}\left(\left(\sup\limits_{|\xi-t|\le \delta}H_n(\xi)\right) - H_{n+1}(t)\right)=-\infty$  
for each $n\in\mathbb N$ and $\delta>0$.
			\end{itemize}
\end{rem}

Let $\mathcal F^t: E'\to H(Q)$ be the adjoint map of $\mathcal F: H(Q)'\to E$
with respect to dual systems $(H(Q)', H(Q))$ and $(E', E)$. 
Then $\mathcal F^t(\varphi)(z)=\varphi\left(e_z\right)$,
\, $z\in Q$, $\varphi\in E'$. In addition, $\mathcal F^t$ is a topological isomorphism of the strong dual
of $E$ onto $H(Q)$ (see \cite[3.2]{CAOT}).
We will write $\widehat\varphi:=\mathcal F^t(\varphi)$ for $\varphi\in E'$.

By $M$ denote the operator of multiplication by the independent variable.

\begin{rem} \label{INTERPOLATING}
	(i) By \cite[Lemma 14]{CAOT} the equality
	$$
	T_{z,g_0}(f)=g_0(z)D_z(M(f))-M(f)(z)D_z(g_0), \,\,z\in\mathbb C, f\in E, 
	$$
	holds.
	
	\noindent
	(ii) For $f\in E\backslash\{0\}$, $h\in H(Q)$ let $\omega_f(z,h)$ be the Leont'ev's interpolating function
	(see \cite{LEONTEV}). 
	Using the equality \cite[Example 1]{Ufa}
	$\omega_f(z,h)={\mathcal F}^{-1}(D_z(f))(h)$, $z\in\mathbb C$, 
		we rewrite the equality in (i) for $\varphi\in E'$ as follows:
	\begin{equation} \label{SOOTNOSHENIE}
		B_\varphi(f)(z)=g_0(z)\omega_{M(f)}(z,\widehat\varphi) - M(f)(z)\omega_{g_0}(z,\widehat\varphi),
		\,\, z\in\mathbb C,\, f\in E\backslash\{0\}.
	\end{equation}
	
	\noindent
	(iii) Let $f\in E\backslash\{0\}$, $\lambda\in\mathbb C$ and $f(\lambda)=0$. By \cite[Lemma 2]{LEONTEV}
	for the function $f_1(t):=\frac{f(t)}{t-\lambda}$ 
	\begin{equation} \label{SOOTNOSHENIE2}
		\omega_f(z,h)=(z-\lambda)\omega_{f_1}(z, h) - \frac{1}{2\pi i}\int\limits_C \gamma_f(t) h(t) dt,
		\,\, z\in\mathbb C,\, h\in H(Q).
	\end{equation}
	Here $C$ is a closed convex curve in $Q$, which surrounds 
	the conjugate diagram of $f$, 
	$\gamma_f$ is the Borel transform of $f$.
	
	\noindent
	(iv) For each $\varphi\in E'$, $f\in E$ the equality
	$\varphi(f)=\frac{1}{2\pi i}\int\limits_C \gamma_f(t)\widehat\varphi(t) dt$ holds,
	where $C$ is a closed convex curve in $Q$, which surrounds 
	the conjugate diagram of $f$.
\end{rem}


The main aim of this article is to describe operators of $\mathcal K(D_{0,g_0})$,
which are an isomorphism of $E$, and to classify operators of
$\mathcal K(D_{0,g_0})$, which are not isomorphism of $E$.

\section{Auxiliary results}

We put  
$B_n:=\{f\in E_n\,|\, \|f\|_n\le 1\}$ and
$\|\varphi\|_n^*:=\sup\limits_{f\in B_n}|\varphi(f)|$, \,$\varphi\in E'$, $n\in\mathbb N$.

For $\varphi\in E'$ the operator
$A_\varphi(f)(z):=\varphi_t\left(t\widetilde T_{z,g_0}(f)(t)\right)$, $z\in\mathbb C$, $f\in E$,
is continuous and linear in $E$.
For each $\varphi\in E'$ the equality
$B_\varphi(f)=\varphi(g_0)f + A_\varphi(f)$, $f\in E$, holds.
It allows to study of properties
$ B_\varphi$, using the theory of compact operators in Banach spaces.
A key to this is the following result.

\begin{lem} \label{compact} Let $g_0\in E_m$ for some $m\in\mathbb N$.
	For each functional $\varphi\in E'$, each $n\ge m$ the operator $A_\varphi$ is compact in $E_n$.
\end{lem} 

\begin{proof}
	The proof is similar to one of V.A. Tkachenko \cite[Theorem 2]{TKACH}.  
	Since the restriction of $\varphi$ on each space $E_k$ is continuous on
	$E_k$, then for all $n\in\mathbb N$, $h\in E_{n+2}$ we have
	$$
	|\varphi(h)|\le\|\varphi\|_{n+2}^*\|h\|_{n+2}.
	$$
	Fix $n\ge m$ and $\varepsilon>0$, $z\in\mathbb C$, $f\in B_n$.
	For $t\in\mathbb C$, such that $|t-z|\ge 1/\varepsilon$, we obtain
	$$
	\frac{|t||f(t)g_0(z)-f(z)g_0(t)|}{|t-z|\exp(H_{n+2}(t))}\le
	\varepsilon\left(\frac{|t||f(t)||g_0(z)|}{\exp(H_{n+2}(t))}
	+\frac{|t||f(z)||g_0(t)|}{\exp(H_{n+2}(t))}\right)\le
	$$
	$$
	\varepsilon(C_1|g_0(z)|+C_2|f(z)|)\le
	\varepsilon\left(C_1\|g_0\|_n\exp(H_{n}(z)) +C_2\exp(H_n(z))\right)=
	$$
	\begin{equation} \label{first}
		\varepsilon(C_1\|g_0\|_n+C_2)\exp(H_n(z)),
	\end{equation}
	where 
	$$
	C_1=\sup\limits_{h\in B_n}\sup\limits_{t\in\mathbb C}\frac{|t||h(t)|}{\exp(H_{n+2}(t))}<+\infty,\,\,
	C_2=\sup\limits_{t\in\mathbb C}\frac{|t||g_0(t)|}{\exp(H_{n+2}(t))}<+\infty. 
	$$
	
	Let now $|t-z|\le 1/\varepsilon$. Applying the maximum modulus principle to the
	holomorphic function $\frac{f(t)g_0(z)-f(z)g_0(t)}{t-z}$, we conclude, that there exists
	$t_0\in\mathbb C$, such that $|t_0-z|=1/\varepsilon$ and
	$$
	\frac{|t||f(t)g_0(z)-f(z)g_0(t)|}{|t-z|\exp(H_{n+2}(t))}\le
	\varepsilon C_3
	\frac{|f(t_0)||g_0(z)|+|f(z)||g_0(t_0)|}{\exp(H_{n+1}(t)))}\le
	$$
	$$
	2\varepsilon C_3\|g_0\|_n\exp(H_{n}(t_0) + H_n(z) -H_{n+1}(t))\le
	$$
	\begin{equation} \label{two}
		2\varepsilon C_3\|g_0\|_n\exp(H_n(z) + \beta(z)),
	\end{equation}
	where 
	$C_3=\sup\limits_{t\in\mathbb C}\left(\exp\left({\rm log}(1+|t|)+H_{n+1}(t)-H_{n+2}(t)\right)\right)<+\infty$
	and
	$$
	\beta(z)=\sup\limits_{|\eta-z|\le 1/\varepsilon}\left(\left(\sup\limits_{|\xi-\eta|\le 2/\varepsilon}H_n(\xi)\right)-
	H_{n+1}(\eta)\right). 
	$$
	From inequalities (\ref{first}), (\ref{two}) and Remark \ref{supportfunction} it follows, that
	$$
	\lim\limits_{|z|\to\infty}\sup\limits_{f\in B_n}\frac{|A_\varphi(f)(z)|}{\exp(H_n(z))}=0.
	$$
	Hence the set $A_\varphi(B_n)$ is relatively compact in $E_n$.	
\end{proof}

\begin{lem} \label{isomorphism} Let $\varphi\in E'$.
	If the operator $B_\varphi:E\to E$ is injective and $\varphi(g_0)\ne 0$,
	then it is a topological isomorphism  $E$ onto $E$.
\end{lem}

\begin{proof} Let $g_0\in E_m$ for some $m\in\mathbb N$.
By Lemma \ref{compact} the operator $A_\varphi$ is compact in each Banach space 
	$E_n$, $n\ge m$.
	Since the equality $B_\varphi(f)=\varphi(g_0)f + A_\varphi(f)$, $f\in E$, holds
	and $\varphi(g_0)\ne 0$, then
	by the Fredholm alternative the restriction of $B_\varphi$ on each space
	$E_n$, $n\ge m$, is a topological isomorphism $E_n$ on itself. From this it follows that
	$B_\varphi: E\to E$ is a topological isomorphism $E$ onto $E$.
\end{proof}

\medskip
In the next part of this section let $g_0=Pe_\lambda$ for some $\lambda\in Q$ and some polynomial $P$
such that $P(0)=1$.
By $\mathcal D(P)$ we denote the set of all polynomials $q$, dividing $P$ and such that
$q(0)=1$.  

We will use a characterization of proper closed $D_{0,g_0}$-invariant subspaces of $E$, 
obtained in \cite[Corollary 20]{AA} and \cite[Theorem 2]{ITOGI17}.

\begin{lem} \cite{AA}, \cite{ITOGI17} \label{INVARIANT}
	For a subspace $S$ of $E$ following assertions are equivalent:
	\begin{itemize}
		\item[(i)] $S$ is a proper closed $D_{0,g_0}$-invariant subspace of $E$.
		
		\item[(ii)] There exists a polynomial $q\in\mathcal D(P)$ of degree greater or equal to $1$, such that $S=q E$, or
		there exist a polynomial $q\in\mathcal D(P)$ and an integer $n\ge 0$, such that
		$n\ge {\rm deg}(P)-{\rm deg}(q)-1$ and
		$S=qe_\lambda\mathbb C[z]_n$.
		
	\end{itemize}
\end{lem}

\begin{lem} \label{KERNEL}
	Let $\varphi\in E'$ and the operator $B_\varphi: E\to E$
	be not injective. Then $B_\varphi(g_0)=0$.
\end{lem}

\begin{proof}	
	For $\varphi=0$ this statement is obviously.
	Let $\varphi\ne 0$ and $S:={\rm Ker}\, B_\varphi$. Then $S$ is a proper closed $D_{0,g_0}$-invariant subspace of $E$.
	We will apply Lemma \ref{INVARIANT}.
	
	If there exists a polynomial $q\in\mathcal D(P)$ of degree greater or equal to $1$ such that $S=qE$,
	then $g_0\in S$.	
	We assume now that there are a polynomial $q\in\mathcal D(P)$, an integer $n\ge 0$ with $n\ge {\rm deg}(P)-{\deg(q)}-1$, 
	for which $S=qe_\lambda\mathbb C[z]_n$. If ${\rm deg}(q)={\rm deg}(P)$, then $q=P$ and $g_0=Pe_\lambda\in S$.
	Consider the case ${\rm deg}(q)<{\rm deg}(P)$. In this case there exists $\lambda\in\mathbb C$, such
	that $P(\lambda)=0$ and 
	$q(z)(z-\lambda)$ divides $P(z)$.
	Note that the degree of the polynomial $P_1(z)=\frac{P(z)}{q(z)(z-\lambda)}$ is equal to
	${\rm deg}(P)-{\deg(q)}-1$. Hence the function $g_1(z)=q(z)e^{\lambda z}P_1(z)=\frac{P(z)}{z-\lambda}e_\lambda$
	belongs to $S$.
	From $B_\varphi(g_1)=0$, by (\ref{SOOTNOSHENIE}), it follows that
	\begin{equation} \label{interpolationequation}
	g_0(z)\omega_{M(g_1)}(z,\widehat\varphi)- M(g_1)\omega_{g_0}(z,\widehat\varphi)=0.
	\end{equation}
	for all $z\in\mathbb C$ and $0=B_\varphi(g_1)(0)=\varphi(g_1)$.
	By Remark \ref{INTERPOLATING}, this implies that $\frac{1}{2\pi i}\int\limits_C \gamma_{g_1}(t)\widehat\varphi(t)dt=0$
	(a closed convex curve $C$ in $Q$ surrounds 
	the conjugate diagram of $g_1$).
	Multiplying (\ref{interpolationequation}) by $z-\lambda$ and using the equality (\ref{SOOTNOSHENIE2}), we infer
	$$
	g_0(z)\omega_{M(g)}(z,\widehat\varphi)- M(g)\omega_{g_0}(z,\widehat\varphi)=0
	$$
	for all $z\in\mathbb C$. Consequently, by (\ref{SOOTNOSHENIE}),
	$B_\varphi(g_0)=0$.
\end{proof}

\begin{lem} \label{injectiv}
		The following assertions are equivalent:
	\begin{itemize}
		\item[(i)] The operator $B_\varphi: E\to E$ is injective.
		\item[(ii)] $\varphi(g_0)\ne 0$.
	\end{itemize}
\end{lem}

\begin{proof}
	(i)$\Rightarrow$(ii): From $B_\varphi(g_0)=\varphi(g_0)g_0$ it follows that
	$\varphi(g_0)\ne 0$.
	
	(ii)$\Rightarrow$(i): Suppose that $B_\varphi$ is not injective.
	By Lemma \ref{KERNEL} $B_\varphi(g_0)=0$ and, consequently,
	$0=B_\varphi(g_0)(0)=\varphi(g_0)$. A contradiction. 
\end{proof}

For $\lambda\in\mathbb C$ and an integer $k\ge 0$, we introduce the functional $\delta_{\lambda, k}(f):=
f^{(k)}(\lambda)$, $f\in E$. All these functionals are continuous and linear on $E$.

\begin{lem} \label{delty} Let ${\rm deg}(P)\ge 1$, $k(\lambda)$ be the multiplicity of
	a zero $\lambda$ of $P$. 
	\begin{itemize}
		\item[(i)] $\delta_{\lambda,k}\otimes\delta_{\mu,l}=0$ for all zeros $\lambda, \mu$
		of $P$ and for all 
		integers $k, l$ with $0\le k\le k(\lambda)-1$, $0\le l\le k(\mu)-1$.
		\item[(ii)] $B_{\delta_{\lambda,k}} B_{\delta_{\mu,l}}=0$
		for all zeros $\lambda, \mu$
		of $P$ and for all 
		integers $k, l$ with $0\le k\le k(\lambda)-1$, $0\le l\le k(\mu)-1$. 
	\end{itemize}
\end{lem}

\begin{proof} The assertion (i) is verified directly (see, for example, \cite[the proof of Lemma 6]{AA}).
	
	The equality in (ii) follows from $B_\varphi B_\psi= B_{\varphi\otimes\psi}$,
	$\varphi, \psi\in E'$ (see Theorem \ref{ALGEBRA}).
\end{proof}

Suppose that ${\rm deg}(P)\ge 1$. For a polynomial $q\in\mathcal D(P)$ of degree greater or equal to $1$
let $\lambda_j$, $1\le j\le m$, be all different zeros of $q$,
$k_j$ be the multiplicity of the zero $\lambda_j$ of $q$. 
We define  the "canonical"\, functional, corresponding to $q$, by
$$
\delta(q):=\sum\limits_{j=1}^m\sum\limits_{k=0}^{k_j-1}\delta_{\lambda_j, k}.
$$

\begin{lem} \label{prodelty}
	Let ${\rm deg}(P)\ge 1$.
	For each polynomial $q\in\mathcal D(P)$ of degree greater or equal to $1$ the equality
	${\rm Ker}\,B_{\delta(q)}= q E$ holds.
\end{lem}

\begin{proof}
Employing standard calculations, for $f\in E$ we obtain:
	$$
	B_{\delta(q)}(f)(z)=	\delta(q)_t\left(f(t)g_0(z) + z\frac{f(t)g_0(z) - f(z)g_0(t)}{t-z}\right)=
	$$
	$$
g_0(z)\sum\limits_{j=1}^m\sum\limits_{k=0}^{k_j-1}f^{(k)}(\lambda_j) +
	zg_0(z)\sum\limits_{j=1}^m\sum\limits_{k=0}^{k_j-1}\sum\limits_{s=0}^k C_{k}^s f^{(k-s)}(\lambda_j)
	\frac{(-1)^{s}s!}{(\lambda_j-z)^{s+1}}=
	$$	
\begin{equation} \label{drobi}
	g_0(z)\sum\limits_{j=1}^m\sum\limits_{s=1}^{k_j}\frac{1}{(\lambda_j-z)^s}
	\sum\limits_{l=0}^{k_j-s}\beta_{s,l}f^{(l)}(\lambda_j),
\end{equation}
	where all constants $\beta_{l,s}$ are independent of $f\in E$
	and $\beta_{s,k_j-s}\ne 0$, $1\le s\le k_j$, $0\le l\le k_j-s$.	
	
Let $B_{\delta(q)}(f)=0$ for some $f\in E$. From (\ref{drobi}) it follows that 
$\sum\limits_{l=0}^{k_j-s}\beta_{s,l}f^{(l)}(\lambda_j)=0$, \,
$1\le s\le k_j$, $1\le j\le m$.
Hence $f^{(l)}(\lambda_j)=0$, $0\le l\le k_j-1$, $1\le j\le m$, and,
consequently, $f\in q E$.
Vice versa, if $f\in q E$, then  (\ref{drobi}) implies $B_{\delta(q)}(f)=0$.
\end{proof}

\section{Main results}

In this section we fix a point $\lambda\in Q$ and a polynomial $P$ with $P(0)=1$
and set $g_0:=Pe_\lambda$.

\begin{thm} \label{ISOMORPHISM} For $\varphi\in E'$
	the following assertions are equivalent:
	\begin{itemize}
		\item[(i)] The operator $B_\varphi$ is a topological isomorphism 
		$E$ onto $E$.
		\item[(ii)] $\varphi(g_0)\ne 0$.
	\end{itemize}
\end{thm}

\begin{proof}
	(i)$\Rightarrow$(ii): If $B_\varphi$ is a topological isomorphism $E$ onto $E$, then the operator $B_\varphi$
	is injective. Hence $\varphi(g_0)\ne 0$ by Lemma \ref{injectiv}.
	
	(ii)$\Rightarrow$(i): By Lemma \ref{injectiv} $B_\varphi$ is injective in $E$.
	By Lemma \ref{isomorphism} $B_\varphi:E\to E$ is a topological isomorphism "onto"\,.
\end{proof}

\medskip
We will prove a result on the factorization of nonzero operators $B_\varphi$.
Note, that the lattice of proper closed $D_{0,g_0}$-invariant subspaces of $E$ 
is not linearly ordered in the case when the function $g_0$ has zeros \cite[Theorem 2]{ITOGI17}.
This significantly affects factorization.

For polynomials $q, r\in\mathcal D(P)$ we denote by $(q,r)_1$ the greatest common 
divisor $d$ of $q$ and $r$ with $d(0)=1$.

\begin{thm} \label{RASCHEPL}Let $\varphi\in E'$, $\varphi\ne 0$ and $\varphi(g_0)=0$. 
	Then either  
	there exist $\psi\in E'$, 
	$n\in\mathbb N$, for which $B_\psi$ is a topological isomorphism $E$ onto $E$
	and $B_\varphi=D_{0, g_0}^n B_\psi$,
	or there are a polynomial $q\in\mathcal D(P)$ of degree greater or equal to $1$, an integer $n\ge 0$, $\psi\in E'$, 
	such that $B_\psi$ is a topological isomorphism $E$ onto $E$
	and $B_\varphi= B_{\delta(q)} D_{0,g_0}^n B_\psi$.	
\end{thm}

\begin{proof} We will exploit Lemma \ref{INVARIANT}.
	First of all, $S={\rm Ker}\,B_\varphi$ is a proper closed $D_{0,g_0}$-invariant subspace of $H(Q)$.
	We suppose that $S=qe_\lambda\mathbb C[z]_n$ for some $q\in\mathcal D(P)$
	and some integer $n\ge 0$, for which $n\ge {\rm deg}(P) - {\rm deg}(q) - 1$. We will show, that $q=P$.
	Assume that ${\rm deg}(q)<{\rm deg}(P)$.  
	Since $\varphi(g_0)=0$, then $B_\varphi(g_0)=\varphi(g_0)g_0=0$.
	Consequently, $g_0\in S$, and hence $Pe_\lambda\mathbb C[z]_0\subset S$.
	Choose the greatest integer $m\ge 0$, such that
	$Pe_\lambda\mathbb C[z]_m\subset S$.
	Since the space ${\rm Ker}\, D_{0,g_0}^{m+1}=Pe_\lambda\mathbb C[z]_m$ is finite-dimensional
	and the operator $D_{0,g_0}^{m+1}: E\to E$ is surjective, then there exists a continuous linear right inverse
	$R: E\to E$ to $D_{0,g_0}^{m+1}$ \cite[Theorem 10.3]{MEIVOGT}. 
	Then $RD_{0,g_0}^{m+1}(f) - f\in {\rm Ker}\, D_{0,g_0}^{m+1}$ for all $f\in E$. 
	Note that $\varphi=0$ on ${\rm Ker}\,B_\varphi$, since
	$\varphi(f)=B_\varphi(f)(0)$ for each $f\in E$.
	Consequently,
	$\gamma D_{0,g_0}^{m+1}=\varphi$ for the functional $\gamma:=\varphi R\in E'$.
	For each $z\in\mathbb C$, $f\in E$ we obtain:
 	$$
	B_\varphi(f)(z)=\varphi(T_{z,g_0}(f))=\gamma\left(D_{0,g_0}^{m+1}\left(T_{z,g_0}(f)\right)\right)=
	\gamma\left(T_{z,g_0}\left(D_{0,g_0}^{m+1}(f)\right)\right)=
	$$
	$$
	B_\gamma\left(D_{0,g_0}^{m+1}(f)\right)(z)= D_{0,g_0}^{m+1} B_\gamma(f)(z),
	$$
	i.\,e. $B_\varphi=D_{0,g_0}^{m+1} B_\gamma$. In addition, $\gamma(g_0)=0$.
	In fact, otherwise $B_\gamma$ is injective by Lemma \ref{injectiv} and
	${\rm Ker}B_\varphi= {\rm Ker}\,D_{0,g_0}^{m+1}=Pe_\lambda\mathbb C[z]_m$. A
	contradiction with ${\rm Ker}\,B_\varphi=qe_\lambda\mathbb C[z]_n$.
	Hence there exist $s\in\mathbb N$ and $\xi\in E'$, for which
	$B_\gamma= D_{0,g_0}^s B_\xi$, and consequently, $B_\varphi= D_{0,g_0}^{m+s+1} B_\xi$.
	This is a contradiction with the maximality of $m$.
	Thus, $q=P$. As $\psi$ we take a functional, defined as $\gamma$ above.
	By Theorem \ref{ISOMORPHISM} $B_\psi$ is a topological isomorphism $E$ onto $E$.
	
	Let now $S=q E$ for a polynomial $q\in\mathcal D(P)$ of degree greater or equal to $1$.
	By Lemma \ref{prodelty} ${\rm Ker} B_{\delta(q)}=q E$. 
	Since ${\rm Ker}\,B_{\delta(q)}$ has the finite codimension, then
	the image $B_{\delta(q)}(E)$ is finite-dimensional.
	From this it follows that
	$B_{\delta(q)}(E)$ is a Fr\'echet space with the topology induced from $E$. 
Consequently, there is a continuous linear right inverse
	$R_0: B_{\delta(q)}(E)\to E$ to $B_{\delta(q)}: E\to B_{\delta(q)}(E)$.
	Define a functional $\xi_0$ on $B_{\delta(q)}(E)$ as $\xi_0:=\varphi R_0$.
	Then $\xi_0$ is continuous and linear on $B_{\delta(q)}(E)$ with the topology,
	induced from $E$.
	By the Hahn-Banach Theorem $\xi_0$  can be extended to a continuous linear functional 
	$\xi$ on $E$.
	Since $R_0 B_{\delta(q)}(f) - f\in {\rm Ker}\,B_{\delta(q)}={\rm Ker}\,B_\varphi$ for all $f\in E$,
	then $\xi_0 B_{\delta(q)}=\varphi$ and also $\xi B_{\delta(q)}=\varphi$. 
As in the first case, from this we infer $B_\varphi= B_{\delta(q)} B_\xi$.
	If $\xi(g_0)\ne 0$, then the lemma is proved (with $\psi=\xi$ and $n=0$). 
	If $\xi(g_0)=0$, then we factorize $B_\xi$ in the form $B_\xi=D_{0,g_0}^n B_\psi$, where $n\in\mathbb N$,
	$\psi\in E'$, $\psi(g_0)\ne 0$, or
	$B_\xi=B_{\delta(r)} B_\tau$, $\tau\in E'$, $r\in\mathcal D(P)$. Second decomposition is not valid, 
	since otherwise $B_\varphi=B_{\delta(q)} B_{\delta(r)} B_\tau=0$ by Lemma \ref{delty}.
	In addition, $B_\psi$ is a topological isomorphism $E$ onto $E$ by 
	Theorem \ref{ISOMORPHISM}.
\end{proof}

\begin{cor}
	Each nonzero operator from $\mathcal K(D_{0,g_0})$, which is not
	finite-dimensional, is surjective and has a continuous linear right inverse
\end{cor}

\begin{rem} Let ${\rm Lat}(D_{0,g_0}, E)$ be the lattice of all closed $D_{0,g_0}$-invariant
	subspaces of $E$. From the proof of Theorem \ref{RASCHEPL} it follows, that 
	the set of kernels of all operators $B_\varphi$, $\varphi\in E'$, coincides with
	${\rm Lat}(D_{0,g_0}, E)$ if and only if the function $g_0=P e_\lambda$ has no zeros, i.\,e. $P\equiv 1$.
\end{rem}

\begin{rem} \label{OBOBINT} V.A. Tkachenko \cite{TKACH} investigated properties
	of the commutant of the operator of generalized integration
	$\mathcal I$ in the strong dual of a countable inductive limit of weighted Banach spaces
	of entire functions, 
	whose growth is determined with
	the help of a 
	$\rho$-trigonometric convex function ($\rho>0$) with values in $(-\infty,+\infty]$.
	The operator $\mathcal I$ is the dual map
	(we use notations of this article)
	to the operator $D_{0,g_0}$ for a function $g_0=e^{\mathcal P}$, where $\mathcal P$ is a polynomial. 
	This function has no zeros in $\mathbb C$. The operator $\mathcal I$ is unicellular. In the unicellular case
	of our article, if $g_0(z)=e^{\lambda z}$, Theorems
	\ref{ISOMORPHISM} and \ref{RASCHEPL} follow from statements, proved by V.A. Tkachenko 
	\cite[\S~4, Property d); Theorem 2]{TKACH}.
	
\end{rem}

\section{The generalized Duhamel product}

We will apply Theorem \ref{ISOMORPHISM} to a multiplication in $H(Q)$.
Let $g_0=Pe_\lambda$, where $\lambda\in Q$ and $P$ is a polynomial such that
$P(0)=1$.
By \cite[\S~4]{ITOGI17} $\mathcal F^t(\varphi\otimes\psi)=
\mathcal F^t(\varphi)\ast\mathcal F^t(\psi)$  for all $\varphi, \psi\in E'$, where
$\ast$ is an associative and commutative multiplication in $H(Q)$. 

For a polynomial $r(z)=\sum\limits_{j=0}^n b_jz^j$ we define
the differential operator $r(D)(f):=\sum\limits_{j=0}^n b_j f^{(j)}$.
Note, that $\varphi(P e_\lambda)=P(D)\left(\mathcal F^t(\varphi)\right)(\lambda)$, $\varphi\in E'$.

Let $m:={\rm deg}(P)\ge 1$. We introduce polynomials $p_j$, $0\le j\le m-1$, for which
$\sum\limits_{j=0}^{m-1}p_j(t) z^j=
\frac{P(t)-P(z)}{t-z}$. 
Set $\widetilde p_j(t):=t p_j(t)$, $0\le j\le m-1$, $t\in\mathbb C$.
As shown in \cite[\S~4]{ITOGI17}, for all $f, h\in H(Q)$, $z\in Q$
$$
(f\ast h)(z)= 
$$
$$
h(\lambda) P(D)(f)(z) +
\int\limits_{\lambda}^z P(D)(f)(\xi) h'(z+\lambda-\xi)d\xi -
\sum\limits_{j=0}^{m-1}\widetilde p_j(D)(f)(z) h^{(j)}(\lambda),
$$
where the integral is taken along the line segment from $\lambda$ to $z$.
Employing integration by parts and substitation 
$\eta=z+\lambda-\xi$, for $f, h\in H(Q)$, $z\in Q$, we infer
$$
(f\ast h)(z)=
$$
$$
P(D)(f)(\lambda)h(z) + \int\limits_{\lambda}^z \left(P(D)(f)\right)'(\eta) h'(z+\lambda-\eta)d\xi -
\sum\limits_{j=0}^{m-1}\widetilde p_j(D)(f)(z) h^{(j)}(\lambda).
$$
This expression of $\ast$ emphasizes the significence of the factor $P(D)(f)(\lambda)$.
For $P\equiv 1$
$$
(f\ast h)(z)=f(\lambda)h(z) + \int\limits_\lambda^z f'(\eta) h(z+\lambda-\eta) d\eta, \,\, z\in Q,\, f, h\in H(Q).
$$
If $P\equiv 1$, $\lambda=0$, then $f\ast h$ is the Duhamel product.  
In the space of all holomorphic functions on a domain in $\mathbb C$, star-shaped with respect to the origin,
this product was investigated at first by N. Wigley \cite{WIGLEY}. 

Define for $f\in H(Q)$ the Duhamel operator $G_f(h):=f\ast h$, $h\in H(Q)$, which
is continuous and linear in $H(Q)$. Note that $\widehat{S_{\varphi}(\psi)}= G_{\widehat\varphi}\left(\widehat\psi\right)$ for all
$\varphi, \psi\in E'$.
Applying standard dual arguments to Theorem \ref{ISOMORPHISM}, we get the following result:

\begin{cor} For $f\in H(Q)$ the operator
	$G_f$ is a topological isomorphism $H(Q)$ onto $H(Q)$ if and only if $P(D)(f)(\lambda)\ne 0$.
\end{cor}

In the case $g_0\equiv 1$, i.\,e. $P\equiv 1$, $\lambda=0$, this statement was proved by N. Wigley
\cite{WIGLEY} (for a domain $Q$, which is star-shaped with respect to the origin).

\noindent
Olga A. Ivanova \\
Southern Federal University, Vorovich Institute of Mathematics, Mechanics and Computer Sciences \\
344090, Russia, Rostov on Don, Mil'chakova St. 8-a\\
e-mail: {\sf neo$_{-}$ivolga@mail.ru}

\bigskip 
\noindent
Sergej N. Melikhov \\
Southern Federal University, Vorovich Institute of Mathematics, Mechanics and Computer Sciences \\
344090, Russia, Rostov on Don, Mil'chakova St. 8-a\\
Southern Institute of Mathematics -- the Affilate of Vladikavkaz Scientific Centre of RAS\\
362027, Vladikavkaz, Vatutina St., 53\\
e-mail: {\sf melih@math.rsu.ru}


\end{document}